\documentclass{birkjour}
\usepackage{ifthen}
\usepackage{fancyhdr}
\usepackage[us,12hr]{datetime} 
\usepackage{exscale}
\usepackage{tabularx}
\usepackage{latexsym}
\usepackage{amsmath,amssymb,amstext,amsthm} 
\usepackage{graphicx,color,epsfig}
\usepackage[table]{xcolor}
\usepackage{graphicx}
\usepackage{float}
\usepackage{verbatim}
\usepackage{enumerate}
\usepackage{tikz}
\usepackage[final]{pdfpages}
\usepackage{algorithm,algorithmic} 
\usepackage{lineno}
\numberwithin{table}{section}
\renewcommand{\theequation}{\arabic{section}.\arabic{equation}}
\numberwithin{algorithm}{section}
\makeindex

\newtheorem{theorem}{Theorem}[section]

\newtheorem{prop}[theorem]{Proposition}

\newtheorem{cor}[theorem]{Corollary}

\newtheorem{remark}[theorem]{Remark}
\newtheorem{conj}[theorem]{Conjecture}

\newtheorem{lemma}[theorem]{Lemma}

\newcommand{\CC}{{\mathbb C}}
\newcommand{\DD}{{\mathbb D}}
\newcommand{\HH}{{\mathbb H}}
\newcommand{\PP}{{\mathbb P}}
\newcommand{\RR}{{\mathbb R}}
\newcommand{\TT}{{\mathbb T}}

\newcommand{\spec}{\operatorname{spec}}

\begin{document}

\title[Strictly Stable Hurwitz Polynomials]{Strictly Stable Hurwitz Polynomials and their Determinantal Representations}
\author[V. Vinnikov]{Victor Vinnikov}\address{Department of Mathematics\\ Ben-Gurion University of the Negev\\ 
Beersheva, Israel.
}\email{vinnikov@bgu.ac.il}

\author[H. J. Woerdeman]{Hugo J. Woerdeman}\address{Department of Mathematics\\ Drexel University\\ 
3141 Chestnut Street\\ Philadelphia, PA 19104, USA. } \email{hugo@math.drexel.edu} \thanks{The research of HJW is supported by National Science Foundation grant DMS 2000037.
}

\date{ }


\begin{abstract}
We establish various certifying determinantal representation results for a polynomial that contains as a factor a prescribed multivariable polynomials that is strictly stable on a tube domain. The proofs use a Cayley transform in combination with the Matrix-valued Hermitian Positivstellensatz developed in [Grinshpan et. al., {\em Oper. Th. Adv.}
Appl., 255 (2016), 123–136].
\end{abstract}



          \maketitle


\section{Introduction}

Hyperbolic polynomials were originally introduced in the study of linear hyperbolic PDEs 
with constant coeeficients in the 1950s. In the last several decades hyperbolic polynomials
and associated hyperbolicity cones came to play a prominent role in several areas,
in particular in optimization where they appear as feasibility sets for hyperbolic programming.
An outstanding open problem, the generalized Lax conjecture, asks whether any hyperbolic polynomial
admits, up to a factor, a positive definite linear determinantal representation certifying its hyperbolicity
and representing the hyperbolicity cone as a spectrahedral cone 
(a linear slice of a cone of positive definite matrices)
which is a feasibility set for semidefinite programming.
In this paper we relate hyperbolic polynomials to complex polynomials that are stable with respect
to a tube domain over the hyperbolicity cone in the complex Euclidean space,
and consider
certifying linear determinantal representations of complex stable polynomials
using tools of multivariable operator theory.
The progress in understanding such determinantal representations
should be of considerable independent value to operator theory, real algebraic geometry, and optimization.

A homogeneous polynomial $p \in \RR[x_0,x_1,\ldots,x_d]$ is called hyperbolic with respect to $a \in \RR^{d+1}$
if $p(a) \neq 0$ and for any $b \in {\mathbb R}^{d+1}$,
the one-variable polynomial $p_b(t)=p(ta-b)$ has only real roots.
Hyperbolic polynomials were first introduced and investigated by G\"{a}rding \cite{Ga51,Ga59}
(who often refers to Petrovsky \cite{Pet45} as a source) 
in the study of linear partial differential equation with constant coefficients,
see also Lax \cite{Lax58} and the later paper \cite{ABG70}.
As was already discovered by G\"{a}rding, there is, quite surprisingly, a convex cone associated
to any hyperbolic polynomial: if $p$ is hyperbolic with respect to $a$ as above,
then the connected component $C$ of $a$ in $\RR^{d+1} \setminus \{x \in \RR^{d+1} \colon p(x)=0\}$
is a convex cone.
$C$ coincides with the set of all $b \in \RR^{d+1}$ 
such that all the roots of the one-variable polynomial $p_b$ are positive.
Furthermore, $p$ is hyperbolic with respect to any $a' \in C$.
$C$ is called a hyperbolicity cone. 
The irreducible hyperbolic polynomial $p$
with a given hyperbolicity cone $C$ can be recovered as the defining polynomial of the Zariski closure
of the boundary of $C$.

Hyperbolicity cones are
interesting due to their appearence at the crossroads between real algebraic geometry and convexity.
They are also attractive feasible sets
for interior point methods in conic programming \cite{NN94,Nem06}: 
$- \log p(x)$ is a logarithmically homogeneous self-concordant barrier function,
and G\" uler \cite{Gu97} has initiated hyperbolic programming 
--- optimizing an affine linear functional over a hyperbolicity cone ---
that was further developed in \cite{BGLS01,Re06} 
(see, e.g., \cite{Sa19} for a recent application).

If $A_0,A_1,\ldots,A_d$ are real symmetric or complex Hermitian matrices such that
$a_0 A_0 + a_1 A_1 + \cdots + a_d A_d > 0$ (here $X > 0$ means that the matrix $X$ is positive definite),
then it is easily seen that the polynomial
\begin{equation} \label{eq:detrep}
p(x) = \det(x_0 A_0 + x_1 A_1 + \cdots + x_d A_d)
\end{equation}
is hyperbolic with respect to $a$, and the corresponding hyperbolicity cone 
is given by a linear matrix inequality (LMI)
\begin{equation} \label{eq:lmi}
C = \left\{x \in \RR^{d+1} \colon x_0 A_0 + x_1 A_1 + \cdots + x_d A_d > 0\right\}.
\end{equation}
i.e., the hyperbolicity cone is a so called spectrahedral cone \cite{RG95}
(the intersection of the cone of positive definite matrices with a linear subspace)
which is the feasible set of a semidefinite program,
see \cite{NN94,VB96,Nem06} and especially \cite{SIG97} for applications in systems and control.

{\em In the case $d=2$, any hyperbolic polynomial admits a positive definite determinantal representation 
\eqref{eq:detrep}} and therefore 
{\em any hyperbolicity cone admits a LMI representation \eqref{eq:lmi}}
(with matrices of the smallest possible size equal to the degree of 
the irreducible polynomial $p$ with hyperbolicity cone $C$);
this was conjectured by Lax \cite{Lax58} and ---
following partial results in \cite{Du83} and \cite{Vi93} --- proved in \cite{HV07} (see also \cite{LPR05}).
It has been an active research topic, 
with several alternate proofs appearing in the intervening years \cite{Vi12,PV13,Ha17,GKVW16,Kn16}
motivated both by the desire to understand better the case $d=2$, including the computational aspects
(see also \cite{PSV11,PSV12}),
and to make progress towards the higher dimensional situation.
However the so called
{\bf generalized Lax conjecture} that {\em any hyperbolicity cone admits a LMI representation \eqref{eq:lmi}}
remains wide open. 
It is equivalent to the following statement:
{\em if $p$ is hyperbolic with respect to $a$ with the corresponding hyperbolicity cone $C$,
there exists a polynomial $q$ which is hyperbolic with respect to $a$ with a hyperbolicity cone
containing $C$ such that $pq$ admits a certifying
positive definite determinantal representation \eqref{eq:detrep}}.
Notice that by the results of \cite{Br11} $q$ cannot be taken in general to be a power of $p$,
and by the results of \cite{NT12} and \cite{Br11},
$q$ cannot be taken to be a product of a power of $p$ and the power of a linear form.

Hyperbolicity is essentially a homogeneous multivariable version of Hurwitz or upper half plane stability
for real polynomials.
More precisely, a homogeneous polynomial
$p \in \RR[z_1,\ldots,z_d]$ is $\HH^d$-stable, i.e., it has no zeroes in $\HH^d$,
where $\HH=\{z \in \CC \colon {\rm Im}\ z > 0\}$ is the upper half plane, if and only if 
it is hyperbolic with a hyperbolicity cone containing the positive orthant $\RR_+^d$. 
Multivariable version of Schur or unit disc stability for complex polynomials
came to play in important role in multivariable operator theory and multidimensional system theory.
In particular, it was shown in \cite{GKVW16a} that
if $p \in \CC[z_1,\ldots,z_d]$, $p(0)=1$, is a Schur stable polynomial which is strictly stable, 
i.e., $p$ has no zeroes on $\overline{\DD^d}$,
where $\DD=\{z \in \CC \colon |z|<1\}$ is the unit disc,
then there exists a Schur stable polynomial $q$ such that 
$pq$ admits a contractive determinantal representation certifying its stability:
\begin{equation} \label{eq:contrdetrep}
p(z)q(z) = \det(I_{|N|} - K Z_N(z)),
\end{equation} 
where $N=(N_1,\ldots,N_d)$, 
$|N|=N_1+\cdots+N_d$,
$Z_N(z) = \bigoplus_{i=1}^d z_i I_{N_i}$,
and $K \in \CC^{|N| \times |N|}$ is a contraction: $||K\| \leq 1$.

An interesting generalization of multivariable Hurwitz stability 
was considered in the recent papers \cite{JT18,JT18a,JTdW19,DGT20+}:
$\Omega_C$-stability, i.e., having no zeroes in $\Omega_C$,
where $\Omega_C \subseteq \CC^d$ is a Siegel domain of the first kind
(a tube domain over a convex cone as a base) \cite{Ho90,SW71,Pia69,Upmeier}, i.e.,   
$\Omega_C = \RR^d + iC$ with $C \subseteq \RR^d$ an (open) convex cone.  
For the special cases $C=\RR^d_+$ and $C=\{Z \in {\rm Herm}_m \colon Z > 0\}$,
where ${\rm Herm}_m$ denotes the real vector space of Hermitian $m \times m$ matrices,
$\Omega_C$-stable polynomials are Hurwitz stable polynomials on $\CC^d$ and
$\mathbb{UHP}_m$-stable (or psd-stable) polynomials on $\CC^{m \times m}$,
where $\mathbb{UHP}_m = \{Z \in \CC^{m \times m} \colon {\rm Im}\ Z > 0\}$
is Siegel's upper half plane \cite{vdGe08}
(${\rm Im}\ Z = \frac{1}{2i}(Z-Z^*)$).

Inspired by the
generalized Lax conjecture, we are exploring in this paper the following conjecture:
{\em let $p \in {\mathbb C}[z_1,\ldots,z_d]$
be $\Omega_C$-stable,
then there exists an $\Omega_C$-stable polynomial $q \in {\mathbb C}[z_1,\ldots,z_d]$
such that $pq$ admits a certifying determinantal representation
\begin{equation} \label{eq:tubedomaindetrep}
pq = \det(A_0+z_1A_1+\cdots+z_dA_d) 
\end{equation}
where ${\rm Im}\ A_0:=\frac{1}{2i}(A_0-A_0^*) \geq 0$ 
and 
$A_1,\ldots,A_d$ are Hermitian matrices satisfying $x_1A_1+\ldots+x_dA_d>0$ for all $x \in C$.} 
We are able to establish this conjecture in Section \ref{standard} for the case of Hurwitz stable polynomials 
that are {\em strictly} stable, which is a requirement on the growth of the absolute value of the polynomial as the variables tend to infinty within the domain of stability (see \eqref{pbbdbelow}).
In other words, we showed that in this case,
there exists a Hurwitz stable polynomial $q$ such that 
$pq$ admits a certifying determinantal representation: 
\begin{equation} \label{eq:hurwitzdetrep}
pq = \det(A_0+z_1A_1+\cdots+z_dA_d) 
\end{equation}
where ${\rm Im}\ A_0:=\frac{1}{2i}(A_0-A_0^*) \geq 0$,
$A_1,\ldots,A_d \geq 0$, $A_1+\ldots+A_d>0$.
A similar representation in the free noncommutative setting has been considered in \cite{Vol19}.
The proof uses the Cayley transform to map the unit disc $\DD$ onto the upper half plane $\HH$
and the certifying contractive determinantal \eqref{eq:contrdetrep} for strictly stable Schur polynomials
(carrying out calculations very similar to an operator-theoretic proof of the Lax conjecture in \cite{Kn16}). In subsequent sections we obtain similar results for 
tube domains over irreducible symmetric homogeneous cones \cite{Loos,Faraut,Upmeier}, namely those corresponding to (i) matrix and Siegel upper halfspaces (Section \ref{section3}); (ii) skew-symmetric matrix halfspaces (Section \ref{sec3a}); (iii) the bivariable Lorents cone (Section \ref{sec4}); (iv) the $n$-variable Lorentz cone (Section \ref{sec6}). Finally, in Section \ref{sec7} we consider the exceptional tube domain with 27 variables. We have a conjectured result, which depends on resolving a polynomial convexity result for the associated bounded domain (see Conjecture \ref{polcon} for details).
We establish some general results on $\Omega_C$-stable polynomials and hyperbolic polynomials in Section \ref{genprop}.

%

\section{Some general properties of polynomials that are stable on a tube domain over a cone} \label{genprop}

Let $C \subseteq \RR^d$ be an open (convex) cone
and let $\Omega_C = \RR^d + iC \subseteq \CC^d$ be the corresponding tube domain.
Let $p \in \CC[z_1,\ldots,z_d]$, $p \neq 0$, 
and let us write $p = r + iq$  with $r,q \in \RR[z_1,\ldots,z_d]$.
For $0 \leq i \leq n = \deg p$ we denote by $p_i$ the degree $i$ homogeneous component of $p$, similarly for $r_i$ and $q_i$.

\begin{theorem}
1. The following are equivalent:
\begin{description}
\item[(i)]
$p$ is $\Omega_C$-stable, i.e., $p(z) \neq 0$ for all $z \in \Omega_C$.
\item[(ii)]
For all $y \in C$ and all $x \in \RR^d$,
the polynomial $p(x+ty) \in \CC[t]$ is Hurwitz stable, i.e.,
it has no zeroes in the upper half plane $\HH$.
\item[(iii)]
For all $y \in C$ and all $x \in \RR^d$,
the polynomials $r(x+ty), q(x+ty) \in \RR[t]$ are real Hurwitz stable, i.e.,
all their zeroes are real, and their zeroes interlace.
\item[(iv)]
The real rational function $r/q$ satisfies ${\rm Im}\ r/q(z) > 0$ for all $z \in \Omega_C$.
\end{description}

2. Assume that $p$ is $\Omega_C$-stable.
Then there exists $c \in \CC$, $c \neq 0$, 
s.t. $cp_n \in \RR[z_1,\ldots,z_d]$ and is a hyperbolic polynomial with 
$C$ contained in the closure of a union of hyperbolicity cones
of $cp_n$. 
We will say that $p$ is normalized if $p_n \in \RR[z_1,\ldots,z_d]$.

3. Assume that $p$ is a normalized $\Omega_C$-stable polynomial.
Then $q_{n-1}$ is a hyperbolic polynomial so that for every $x\in {\mathbb R}^d$ and $y \in C$ the zeroes of $p_n(x+ty)$ and $q_{n-1}(x+ty) \in \RR[t]$ are all real and interlace.
\end{theorem}

Note that $p_n$ is also called the {\em initial form} of $p$; see, e.g., \cite{JT18a}. In addition, item 2 above was proven in \cite[Theorem 3.1]{DGT20+} for the case for when $p(x)=\det( A_0 + x_1 A_1 + \cdots + x_d A_d)$ with $A_i$ Hermitian, and $\sum_{j=1}^n e_jA_j >0$ for some real $e_1,\ldots , e_n$. We refer to \cite{KPV} for more information on interlacing of hyperbolic polynomials.

\begin{proof}
1. 
{\bf(i)}  $\implies$ {\bf(ii)}:
Assume that $p$ is $\Omega_C$-stable and let $x \in \RR^d$, $y \in C$.
Then ${\rm Im}(x+ty) = {\rm Im}(t) y \in C$ for ${\rm Im}(t)>0$, 
and therefore $p(x+ty) \neq 0$ for ${\rm Im}(t)>0$. 

{\bf(ii)} $\implies$ {\bf(i)}: Obvious --- 
for $z \in \Omega_C$, 
write $z=x+iy$ with $x \in \RR^d$, $y \in C$,
and use {\bf(ii)} with $t=i$.

The equivalence of {\bf(iii)} and {\bf(iv)} 
now follows from the classical results on Hurwitz stable polynomials; see, for instance, \cite{KN}.

2. Assume that 
$p = p_n + p_{n-1} + \cdots +p_0$
is $\Omega_C$-stable. 
Let us consider the homogenization of $p$,
\begin{multline*}
P(\mu_0,\mu_1,\ldots,\mu_d) = \mu_0^n p(\mu_1/\mu_0,\ldots,\mu_d/\mu_0)
\\
= p_n(\mu_1,\ldots,\mu_d) + \mu_0 p_{n-1}(\mu_1,\ldots,\mu_d) + \cdots + \mu_0^n p_0.
\end{multline*}
It follows that $P(\mu_0,z) \neq 0$ for all $\mu_0>0$, $z \in \Omega_C$,
and therefore by 1., for all $\mu_0 > 0$, $x \in \RR^d$, $y \in C$,
all the zeroes of $P(\mu_0,x+ty) \in \CC[t]$ are in the closed lower half plane. Since the zeroes of the polynomial depend continuously on the coefficients, by taking $\mu_0 \to 0+$ it follows that all the zeroes of
$P(0,x+ty) = p_n(x+ty) \in \CC[t]$ are in the closed lower half plane (and therefore $p_n$ is $\Omega_C$-stable).

Assume that $x \in C$ as well. Then also all the zeroes of $p_n(y+tx) \in \CC[t]$ are in the closed lower half plane. Assume that $p_n(x), p_n(y) \neq 0$, so that $p_n(x+ty), p_n(y+tx) \in \CC[t]$ are of degree $n$ and $t=0$ is not a root of either polynomial, 
and notice that since $p_n(y+tx) = t^n p_n(x+t^{-1}y)$ the roots of the two polynomials are inverses of each other. As ${\rm Im}\ t > 0 \implies {\rm Im}\ t^{-1} < 0$, we conclude that for $x,y \in C$, $p(x), p(y) \neq 0$, all the roots of $p(x+ty) \in \CC[t]$ are real. Writing $p_n(x+ty) = p_n(y) t^n + \cdots + p_n(x)$, it follows that $p_n(x) / p_n(y)$ is real. Fixing $y$, we see that the polynomial $p_n(z) / p_n(y) \in \CC[z_1,\ldots,z_d]$ is real on an open set of $\RR^d$, and is therefore real valued. 

We assume now that $p$ is normalized so that $p_n \in \RR[z_1,\ldots,z_d]$ and let again $y \in C$, $p(y) \neq 0$. Then for all $x \in \RR^d$, $p_n(x+ty) \in \RR[t]$ has no roots in the upper half plane and therefore all the roots are real and $p_n$ is hyperbolic with respect to $y$. 

3. Similarly to the proof above, we consider the homogenizations of $r$ and of $q$,
\begin{align*}
R(\mu_0,\mu_1,\ldots,\mu_d) = \mu_0^n &  r(\mu_1/\mu_0,\ldots,\mu_d/\mu_0) \\
& = r_n(\mu_1,\ldots,\mu_d) + \mu_0 r_{n-1}(\mu_1,\ldots,\mu_d) + \cdots + \mu_0^n r_0,\\
Q(\mu_0,\mu_1,\ldots,\mu_d) = \mu_0^n & q(\mu_1/\mu_0,\ldots,\mu_d/\mu_0) & \\
& = q_n(\mu_1,\ldots,\mu_d) + \mu_0 q_{n-1}(\mu_1,\ldots,\mu_d) + \cdots + \mu_0^n q_0.
\end{align*}
Since $p$ is normalized, we have $r_n=p_n$ and $q_n=0$.
By 1., for all $\mu_0  > 0$, $x \in \RR^d$, $y \in C$,
the zeroes of $R(\mu_0,x+ty), Q(\mu_0,x+ty) \in \RR[t]$ are all real and interlace. 
Taking $\mu \to 0+$, we conclude that the zeroes of $R(0,x+ty)= p_n(x+ty), Q(0,x+ty)=q_{n-1}(x+ty) \in \RR[t]$ are all real and interlace. 
\end{proof}

\begin{remark}\rm 
In case the hyperbolic polynomial $p_n$ has a single hyperbolicity cone up to sign (i.e., it has exactly two hyperbolicity cones which are reflections of each other), item 2. means simply that $p_n$ is hyperbolic with hyperbolicity cone containing $C$. This is always the case if the corresponding real projective variety $\left\{[x] \in \PP^{d-1}_\RR \colon p_n(x)=0\right\}$ (where $\PP^{d-1}_\RR$ denotes the $(d-1)$-dimensional real projective space and $[x]$ denotes the point with projective coordinates $x$) is smooth \cite{HV07} or if $p_n$ is irreducible \cite{Kum19}.
\end{remark}

\section{Standard Strictly Stable Hurwitz Polynomials}\label{standard}

Let ${\mathbb H}= \{ z \in {\mathbb C} : {\rm Im} \ z >0 \}$. 

\begin{theorem}\label{phasdetrep}
Let $p(z_1,\ldots, z_d)$ be a polynomial of multidegree $(n_1,\ldots, n_d)$ so that for some $\epsilon >0$ we have
\begin{equation}\label{pbbdbelow} |p(z_1,\ldots, z_d) | \ge \epsilon \prod_{j=1}^d |z_j +i |^{n_j}, \ z_1, \ldots, z_d \in {\mathbb H}. \end{equation} Then there exists a polynomial $q(z_1,\ldots, z_d)$ so that the product $pq$ has a certifying determinantal representation; in other words, there exist matrices $A_0, A_1, \ldots , A_d$ so that 
$$ {\rm Im}\ A_0 = \frac{1}{2i}(A_0-A_0^*) > 0, \  \ A_j\ge 0, j=1,\ldots,d, \ \   A_1+\cdots + A_d >0, $$
$$ p(z_1,\ldots, z_d) q(z_1,\ldots, z_d) = \det( A_0 + z_1A_1 + \cdots + z_d A_d ). $$
\end{theorem}

Let $\DD = \{ z \in \CC : |z|<1\}$ and $\TT = \{ z \in \CC : |z|=1\}$.
We use the Cayley transform in a similar way as in \cite{Kn16}. Recall that the map
$$ \phi(z) = i \frac{1+z}{1-z} $$ maps the unit disk conformally onto the upper half plane sending $\TT$ to $\RR \cup \{ \infty \}$. The inverse is given by
$$ \phi^{-1}(w)= \frac{w-i}{w+i}. $$

\begin{proof}
Put
$$ \tilde{p} (z_1,\ldots, z_d) = p(i \frac{1+z_1}{1-z_1}, \ldots , i \frac{1+z_d}{1-z_d})\prod_{j=1}^d(1-z_j)^{n_j} . 
$$
Then, for $z_j\in\DD$, we have
$$ |\tilde{p} (z_1,\ldots, z_d)| = |p(i \frac{1+z_1}{1-z_1}, \ldots , i \frac{1+z_d}{1-z_d})\prod_{j=1}^d(1-z_j)^{n_j}  |\ge $$
$$  \epsilon \prod_{j=1}^d | i \frac{1+z_j}{1-z_j} +i |^{n_j}\prod_{j=1}^d|(1-z_j)^{n_j}| =\epsilon \prod_{j=1}^d |i(1+z_j)+i (1-z_j)|^{n_j} = \epsilon 2^{\sum n_j}. $$ Thus $\tilde{p}$ is strongly stable on $\overline{\DD}^d$. By applying \cite[Theorem 4.1]{GKVW16a} to $\tilde{p}(rz)$ for an appropriate $r>1$, there exists a polynomial $\tilde{q}$, nonnegative integers $N_j$, $j=1,\dots,d$, and a strictly contractive matrix $K$ of size $\sum N_j \times \sum N_j$ so that 
$$\tilde{p}\tilde{q} (z_1,\ldots, z_d) = \tilde{p} (z_1,\ldots, z_d) \tilde{q} (z_1,\ldots, z_d) = \tilde{p}(0) \tilde{q}(0) \det(I-KZ_N),$$ where
$$ Z_N= {\rm diag} (z_j I_{N_j})_{j=1}^d. $$
Then, for $z_1, \ldots , z_d \in \HH$, 
$$ \tilde{p}\tilde{q} (\frac{z_1-i}{z_1+i},\ldots, \frac{z_d-i}{z_d+i}) \prod_{j=1}^d (z_j+i)^{N_j} = $$ 
\begin{equation}\label{ff} \tilde{p}(0) \tilde{q}(0) \det(\left(\oplus_{j=1}^d (z_j+i)I_{N_j}\right)-K\left(\oplus_{j=1}^d (z_j-i)I_{N_j}\right))= \end{equation}
$$=\tilde{p}(0) \tilde{q}(0) \det(\left(\oplus_{j=1}^d z_jI_{N_j}\right) +i(I+K)(I-K)^{-1} ) \det(I-K) = $$ $$\tilde{p}(0) \tilde{q}(0)\det(A_0+z_1A_1+\cdots + z_d A_d)\det(I-K), $$ where 
$$ A_0 = i(I+K)(I-K)^{-1}, \ A_j = \oplus_{k=1}^d \delta_{jk} I_{N_k}, j=1,\ldots, k, $$
where $\delta_{jk}$ is the Kronecker delta. We have that
$$ \frac{A_0-A_0^*}{2i} = (I-K^*)^{-1} (I-K^*K)(I-K)^{-1} > 0, $$
and since $A_1, \ldots, A_d$ are orthogonal projections that add up to the identity, the matrices $A_j$ are as desired. Finally, one calculates that 
$$ \tilde{p} (\frac{z_1-i}{z_1+i},\ldots, \frac{z_d-i}{z_d+i}) \prod_{j=1}^d (z_j+i)^{n_j} = {p} (z_1,\ldots, z_d) (2i)^{\sum n_j}. $$ One now lets
$$ q(z_1, \ldots , z_d) = \frac{(2i)^{\sum n_j}}{\tilde{p}(0) \tilde{q}(0)\det(I-K)} \ \tilde{q} (\frac{z_1-i}{z_1+i},\ldots, \frac{z_d-i}{z_d+i}) \prod_{j=1}^d (z_j+i)^{N_j-n_j}. $$
If $q$ is a polynomial, we are done. If $q$ still has some factors $z_j+i$ in the denominator, one can increase $N_j$ in \eqref{ff} accordingly, and add corresponding zero rows and columns in $K$.
\end{proof}

\begin{prop} Let $p(z_1,\ldots, z_d)$ be a polynomial of multidegree $(n_1,\ldots, n_d)$. The following are equivalent:
\begin{itemize}
\item[(i)] There exists an $\epsilon >0$ so that \eqref{pbbdbelow} holds.
\item[(ii)] There exists an $\epsilon >0$ so that $|\tilde{p}(z_1,\ldots , z_d) | \ge \epsilon$, $z_1, \ldots, z_d \in {\mathbb D}$.
\item[(iii)] The polynomial $p$ contains a nontrivial term $\prod_{j=1}^d z_j^{n_j}$ and there exists an $\epsilon >0$ so that $|{p}(z_1,\ldots , z_d) | \ge \epsilon$, $z_1, \ldots, z_d \in {\mathbb H}$.
\end{itemize}
\end{prop}

\begin{proof} The equivalence of (i) and (ii) follows from the calculations in the proof of Theorem \ref{phasdetrep}

Assume that (i) holds. Then it follows from \eqref{pbbdbelow} that the total degree of $p$ must at least be $\sum_{j=1}^d n_j$. But the only way that is possible is when $p$ has a nontrivial term $\prod_{j=1}^d z_j^{n_j}$.

It remains to prove that (iii) implies (i). If $p$ contains a nontrivial term $\prod_{j=1}^d z_j^{n_j}$, then this term being the single one of total degree $\sum_{j=1}^d n_j$, will eventually dominate all other terms. Thus, there exists an $M$ and an $\hat\epsilon>0$ so that 
$$ |p(z_1, \ldots , z_d )| \ge\hat\epsilon \prod_{j=1}^d |z_j +i |^{n_j}, \hbox{ when } |z_1|+ \cdots + |z_d| \ge M \hbox{ and }\ z_1, \ldots, z_d \in \HH. $$ Since $|{p}(z_1,\ldots , z_d) | \ge \tilde\epsilon$, $z_1, \ldots, z_d \in {\mathbb H}$, we have that 
$$
\frac{|{p}(z_1,\ldots , z_d)|}{\prod_{j=1}^d |z_j +i |^{n_j}}  \ge \epsilon', $$
 when $|z_1|+ \cdots+ |z_d| \le  M $ and $ z_1, \ldots, z_d \in \HH. $ Now (i) is satisfied with the choice $\epsilon = \min \{\hat\epsilon, \epsilon' \} $. 
\end{proof}

\section{Strict Stability with respect matrix and Siegel upper halfspaces}\label{section3}

We consider the matrix and Siegel upper halfspaces \cite{Siegel}:
$$ \mathbb{UHP}_n = \{ Z \in {\mathbb C}^{n\times n} : {\rm Im}\ Z >0 \}, \mathbb{SUHP}_n=\{ Z=Z^T \in {\mathbb C}^{n \times n} : {\rm Im}\ Z >0 \}.$$
We let our domain be
$$ \mathcal{D}= \mathbb{UHP}_{l_1} \times \cdots \mathbb{UHP}_{l_k} \times \mathbb{SUHP}_{s_1} \times \cdots \times \mathbb{SUHP}_{s_r}.$$
Thus there are
$$ d=l_1^2+\cdots + l_k^2+ \frac12 s_1(s_1+1) + \cdots + \frac12 s_r (s_r+1)$$
complex variables. We write $Z\in \mathcal{D}$ as
$$ Z=Z_1\times \cdots \times Z_k \times W_1 \times \cdots \times W_r,  $$ where
$$ Z_q = (z_{ij}^{(q)})_{i,j=1}^{l_q},  q=1,\ldots, k,\ \ \ W_\rho=(w_{ij}^{(\rho)})_{i,j=1}^{s_\rho}, \rho=1,\ldots, r.$$
Here $w_{ij}^{(\rho)}=w_{ji}^{(\rho)}$ for all $i,j$ and $\rho$. It will be convenient to write a polynomial $p$ in $d$ variables considered on the domain $\mathcal{D}$ as $p(Z)=p(Z_1, \ldots, Z_k, W_1, \ldots , W_r)$.

\begin{theorem}\label{pdetrep}
Let $p(Z)$ be a polynomial in the variables $(z_{ij}^{(q)})_{i,j=1}^{l_q},  q=1,\ldots, k$, $ (w_{ij}^{(\rho)})_{1\le i\le j\le s_\rho}, \rho=1,\ldots, r,$
and let $t_q$, $q=1,\ldots ,k$, denote the total degree of $p$ in the variables $(z_{ij}^{(q)})_{i,j=1}^{l_q}$ and $u_\rho$, $\rho=1,\ldots , r$, the total degree of $p$ in the variables $(w_{ij}^{(\rho)})_{1\le i\le j\le s_\rho}$.
Assume that for some $\epsilon >0$ we have
$$ |p(Z)|=|p(Z_1, \ldots, Z_k, W_1, \ldots , W_r) | \ge \ \ \ \ \ \ \ \ \ \ \ \ \ \ \ \ \ \ \ $$ \begin{equation}\label{pboundedbelow} \epsilon \prod_{q=1}^k  |\det(Z_q+iI)|^{t_q} \prod_{\rho=1}^r  |\det(W_\rho+iI)|^{u_\rho}, Z \in \mathcal{D}. \end{equation} Then there exists a polynomial $q(Z)$ so that the product $pq$ has the certifying determinantal representation
$$ p(Z)q(Z)= \det \left( A_0 + \begin{bmatrix} {\rm diag} (Z_q \otimes I_{N_q})_{q=1}^k & 0 \cr 0 & {\rm diag} (W_\rho \otimes I_{M_\rho})_{\rho =1}^r \end{bmatrix} \right) , $$
with $A_0$ satifying ${\rm Im}\ A_0 >0$ and $N_1,\ldots, N_k, M_1, \ldots, M_r \in {\mathbb N}$.
\end{theorem}

\begin{proof}
Put
$$ \tilde{p} (Z_1,\ldots, Z_k,W_1, \ldots , W_r) = $$ $$ p(\phi(Z_1), \ldots, \phi(Z_k),\phi(W_1), \ldots , \phi(W_r))\prod_{q=1}^k \det(I-Z_q)^{t_q} \prod_{\rho=1}^r \det(I-W_\rho)^{u_\rho}. 
$$
Observe that $\phi(W^T)=\phi(W)^T$, and thus the polynomial $\tilde{p}$ is viewed on a domain $\mathcal B$ consisting of products of Cartan I and Cartan III domains, and due to condition \eqref{pboundedbelow}
one finds that $\tilde{p}$ is strongly stable on $\mathcal B$. One may now apply \cite[Theorem 4.1]{GKVW16a} similarly as in the proof of Theorem \ref{phasdetrep} and obtain a polynomial $\tilde{q}$, a strictly contractive matrix $K$ and $N_1,\ldots, N_k, M_1, \ldots, M_r \in {\mathbb N}$  so that 
$$\tilde{p}\tilde{q} (Z) = \tilde{p}\tilde{q} (0) \det\left(I-K\begin{bmatrix} {\rm diag} (Z_q \otimes I_{N_q})_{q=1}^k & 0 \cr 0 & {\rm diag} (W_\rho \otimes I_{M_\rho})_{\rho =1}^r \end{bmatrix}\right).$$ 
Continuing as in Theorem \ref{phasdetrep} one finds (with $A_0=\phi(K)$) the desired representation.
\end{proof}

\section{Strict Stability: the skew-symmetric case}\label{sec3a}

We let ${\mathcal{SK}}_{2n}=\{ Z \in {\mathbb C}^{2n\times 2n}: Z=-Z^T \}$ denote the $2n \times 2n$ skew-symmetric matrices. Let $J=\begin{bmatrix}
0 & I_n \cr -I_n & 0 
\end{bmatrix}$, which belongs to ${\mathcal{SK}}_{2n}$.
\begin{prop} Let $\psi(Z)= i  (-J+Z)(I-JZ)^{-1}.$ Then $\psi$
is a bijective map from $\{Z \in {\mathcal{SK}}_{2n} : \| Z \| < 1\} $ to $\widetilde{\mathcal{D}}=\{ W\in {\mathcal{SK}}_{2n} : {\rm Im} \ JW >0 \}$, and $\psi^{-1}(W)=J(W+iJ)(iJ-W)^{-1}$.
\end{prop}

\begin{theorem}\label{pdetrepskew}
Let $p(Z)$ be a polynomial in the variables $(z_{ij})_{1\le i < j\le r}$
and let $t$ denote the total degree of $p$.
Assume that for some $\epsilon >0$ we have
\begin{equation}\label{pboundedbelowskew} |p(Z)| \ge \epsilon |\det(Z-iJ)|^{t} , Z \in \widetilde{\mathcal{D}}. \end{equation} Then there exists a polynomial $q(Z)$ so that the product $pq$ has the certifying determinantal representation
$$ p(Z)q(Z)= \det \left( A_0 + (ZJ\otimes I_N) \right) , $$
with $A_0$ satifying ${\rm Im}\ A_0 >0$ and $N \in {\mathbb N}$.
\end{theorem}

\begin{proof}
Put
$$ \tilde{p} (Z) =  p(\psi(Z)) \det(I-JZ)^{t} , Z \in \widetilde{\mathcal{D}}. 
$$
Observe that $\psi(Z^T)=-\psi(W)^T$, and thus the polynomial $\tilde{p}$ acts on a Cartan II domain. Moreover, due to condition \eqref{pboundedbelowskew},
we have $$ |\tilde{p}(Z)| = |p(\psi(Z)) \det(I-JZ)^{t}| \ge \epsilon |\det(\psi(Z) -iJ)^t \det(I-JZ)^{t}|=$$ 
$$ \epsilon |\det(i(-J+Z) -iJ(I-JZ))^{t}| =\epsilon 2^{2nt}, Z \in \widetilde{\mathcal{D}}.$$ Thus $\tilde{p}$
is strongly stable on $\widetilde{\mathcal{D}}$. One may now apply \cite[Theorem 4.1]{GKVW16a} similarly as in the proof of Theorem \ref{phasdetrep} and obtain a polynomial $\tilde{q}$, a strictly contractive matrix $K$ and $N \in {\mathbb N}$  so that 
$$\tilde{p}\tilde{q} (Z) = \tilde{p}\tilde{q} (0) \det\left(I-K (Z\otimes I_{N})\right).$$ 
Continuing as in Theorem \ref{phasdetrep} one finds the desired representation with $A_0=i(I+(J\otimes I_N)K)^{-1}(I-(J\otimes I_N)K)$.
\end{proof}

\section{Strict Stability with respect to the bivariable Lorentz cone}\label{sec4}

We consider the tube domain $$T_{{\mathcal C}_n}= {\mathbb R}^n + i {\mathcal C}_n, {\mathcal C}_n = \{ x \in {\mathbb R}^n : x_1>0, x_1^2-x_2^2 - \cdots -x_n^2 >0 \}. $$
The cone ${\mathcal C}_n$ is known as the Lorentz cone.

\begin{theorem}\label{pdetrep}
Let $p(w_1,w_2)$ be a polynomial of multidegree $(n_1, n_2)$ so that for some $\epsilon >0$ we have
that $(w_1, w_2)  \in T_{{\mathcal C}_2}$ implies \begin{equation}\label{pbbdbelow2} |p(w_1, w_2) | \ge \epsilon  \left|\left(1-\frac{1+w_1^2-w_2^2}{(w_1+i)^2-w_2^2} \right)^2 + \left(\frac{-2w_2}{(w_1+i)^2-w_2^2} \right)^2 \right|^{-n_1-n_2}. \end{equation} Then there exists a polynomial $q(w_1, w_2)$ so that the product $pq$ has a certifying determinantal representation; that is, there exists $k\in{\mathbb N}$ and a $2k\times 2k$ matrix $A_0$ with ${\rm Im}\ A_0 >0$ so that
$$ p(w)q(w) = \det \left( A_0 + \begin{bmatrix}
    w_1 & -iw_2\cr iw_2 & w_1
\end{bmatrix} \otimes I_k \right). $$
\end{theorem}

We will prove the above result, by first proving a determinantal representation theorem on the 2-variable Lie ball, and subsequently using that $T_{{\mathcal C}_2}$ is the Cayley transform of the 2-variable Lie ball.

Let us first recall what the Lie ball is. 
We denote $\| z\|= \sqrt{\sum_{j=1}^n |z_j|^2}$. The {\em Lie ball} in $n$ variables, which is a Cartan type IV domain, is given by
$$ L_n =\{ z \in {\mathbb C}^n : \|z\|^2 + \sqrt{\|z\|^4 - |\sum_{i=1}^n z_i^2|^2} < 1 \},  $$ or, equivalently, 
$$ L_n =\{ z \in {\mathbb C}^n : |\sum_{j=1}^n z_j^2 | <1 \ \hbox{and} \ 1-2\|z\|^2 + |\sum_{j=1}^n z_j^2 |^2>0 \}= $$
$$\{ z \in {\mathbb C}^n : \| z \|  <1 \ \hbox{and} \ 1-2\|z\|^2 + |\sum_{j=1}^n z_j^2 |^2>0 \}.$$

\begin{lemma}
We have that $z\in L_n$ if and only if
$$ M(z):= \|z\|^2 I_n + zz^*-\overline{z}z^T <I .$$ In particular, when $n=2$ we have that $(z_1, z_2)\in L_2$ if and only if $$ \| \begin{bmatrix} z_1 & -z_2 \cr z_2 & z_1 \end{bmatrix} \| <1.$$
\end{lemma}

\begin{proof}
We first note that
$$zz^*-\overline{z}z^T = \begin{bmatrix}
    z & \overline{z} 
\end{bmatrix} \begin{bmatrix}
    1 & 0 \cr 0 & -1 
\end{bmatrix}\begin{bmatrix}
    z^* \cr z^T 
\end{bmatrix} $$ has the same nonzero eigenvalues as
$$ \begin{bmatrix}
    1 & 0 \cr 0 & -1 
\end{bmatrix}\begin{bmatrix}
    z^* \cr z^T 
\end{bmatrix} \begin{bmatrix}
    z & \overline{z} 
\end{bmatrix} = \begin{bmatrix}
    z^*z & z^*\overline{z} \cr -z^Tz & -z^T \overline{z}
\end{bmatrix},$$
which is a $2\times 2$ matrix with zero trace and determinant equal to
$$ -\| z \|^4 + |z^Tz|^2.$$
Thus the eigenvalues are $\pm \sqrt{\| z \|^4 - |z^Tz|^2}$, where we observe that by the Cauchy-Schwarz inequality we have that $|z^Tz|^2 \le \|z \|^2 \| \overline{z}\|^2 = \| z\|^4$. Clearly, also
$$ \| z\|^2 \ge \sqrt{\| z \|^4 - |z^Tz|^2},$$ and thus $M(z)$ is positive semidefinite with largest eigenvalue equal to $\| z\|^2 + \sqrt{\| z \|^4 - |z^Tz|^2}$. The result now follows.

Next, when $n=2$, we get that 
$$ M(z) = \begin{bmatrix} |z_1|^2+|z_2|^2 & z_1\overline{z_2} - \overline{z_1} z_2 \cr \overline{z_1} z_2 -z_1\overline{z_2} & |z_1|^2+|z_2|^2  
    
\end{bmatrix} = \begin{bmatrix} z_1 & -z_2 \cr z_2 & z_1 \end{bmatrix}^* \begin{bmatrix} z_1 & -z_2 \cr z_2 & z_1 \end{bmatrix}, $$
finishing the proof.
\end{proof}

By the above lemma, we have that for $n=2$ the Lie ball domain can be represented as $I-P_-(z)^* P_-(z)>0$, and thus we can apply \cite[Theorem 4.1]{GKVW16a} and obtain the following.

\begin{theorem}\label{2ball}
Let $p(z_1, z_2)$ have no roots on the closed two-variable Lie ball $\overline{L_2}$. Then there exists a polynomial $q(z_1,z_2)$, a positive integer $n_1$ and a $2n_1 \times 2n_1$ strictly contractive matrix $K$ so that 
$$ p(z_1,z_2) q(z_1,z_2) = \det(I_{2n_1} - K (P(z_1,z_2) \otimes I_{n_1})), $$
where $$ P(z_1,z_2) = \begin{bmatrix}
    z_1  & -z_2 \cr z_2 & z_1  
\end{bmatrix}.$$
\end{theorem}

\begin{proof} As $p(z_1, z_2)$ has no roots on the closed two-variable Lie ball $\overline{L_2}$ and $\overline{L_2}$ is compact, there exists an $r>1$ so that $p(z_1/r, z_2/r)$ also has no roots on the closed two-variable Lie ball $\overline{L_2}$. Apply now \cite[Theorem 4.1]{GKVW16a} to $p(z_1/r, z_2/r)$, giving a $q$, $n_1$ and a contraction $K$. Now $K/r$ is the strict contraction that works for $p(z_1, z_2)$.
\end{proof}

Note that in the case of the bitorus a version of Theorem \ref{2ball} holds where $q\equiv \frac{1}{p(0,0)}$; see, e.g., \cite[Theorem 2.1]{GKVW16}. One may ask whether we may choose $q$ also in this way in the above result. 

We would like to use Theorem \ref{2ball} to prove Theorem \ref{pdetrep}. For this we need to establish the following connection between $L_2$ and $T_{{\mathcal C}_2}.$

\begin{lemma}\label{philemma} Define $\Phi_2$ via
$$ \Phi_2(z_1,z_2):= (
 \frac{i(1-z_1^2-z_2^2)}{(1-z_1)^2+z_2^2} , \frac{-2iz_2}{(1-z_1)^2+z_2^2}),  $$
 whenever $z_1,z_2\in \CC$ satisfy $(1-z_1)^2+z_2^2\neq 0$.
In addition, let
$$ P(z)=\begin{bmatrix}
z_1 & -z_2 \cr z_2 & z_1
\end{bmatrix},$$ and assume that $I-P(z)$ is invertible. Then
$$ 
(w_1,w_2)= \Phi_2(z_1,z_2) 
\ \Longleftrightarrow \ \phi (P(z)) = \begin{bmatrix}
    w_1 & -iw_2\cr iw_2 & w_1
\end{bmatrix}. $$
Moreover, $\Phi_2: L_2 \to T_{{\mathcal C}_2}$ is a bijection.
\end{lemma}

\begin{proof}
We compute
$$ \phi(P(z))= i (I+P(z))(I-P(z))^{-1}  $$ $$=i \begin{bmatrix}
1+z_1 & -z_2 \cr z_2 & 1+z_1
\end{bmatrix} \begin{bmatrix}
1-z_1 & z_2 \cr -z_2 & 1-z_1
\end{bmatrix}^{-1}$$
$$=\frac{i}{(1-z_1)^2+z_2^2} \begin{bmatrix}
1+z_1 & -z_2 \cr z_2 & 1+z_1
\end{bmatrix} \begin{bmatrix}
1-z_1 & -z_2 \cr z_2 & 1-z_1
\end{bmatrix}=$$ $$\frac{1}{(1-z_1)^2+z_2^2} \begin{bmatrix}
i(1-z_1^2-z_2^2) & -2iz_2 \cr 2iz_2 & i(1-z_1^2-z_2^2)
\end{bmatrix} =\begin{bmatrix}
    w_1 & -iw_2\cr iw_2 & w_1
\end{bmatrix}$$
if and only if $(w_1,w_2)= \Phi_2(z_1,z_2)$,
proving the first part.

By properties of the Cauchy transform we obtain that $I-P(z)^*P(z) >0 $ if and only if $\begin{bmatrix}
    w_1 & -iw_2\cr iw_2 & w_1
\end{bmatrix}=\phi(P(z))$ has positive imaginary part. Finally, we have that
\begin{equation}\label{imW} {\rm Im}\ \begin{bmatrix}
    w_1 & -iw_2\cr iw_2 & w_1
\end{bmatrix} = \begin{bmatrix}
{\rm Im}\ w_1 & -i{\rm Im}\ w_2 \cr i{\rm Im}\ w_2 & {\rm Im}\ w_1 
\end{bmatrix}>0 \end{equation} if and only ${\rm Im}\ w_1 >0$ and 
$$({\rm Im}\ w_1)^2 - ({\rm Im}\ w_2)^2 >0.$$ These are exactly the conditions that give $\begin{bmatrix}
    w_1 \cr w_2
\end{bmatrix} \in T_{{\mathcal C}_2}$.
\end{proof}

We are now able to prove Theorem \ref{pdetrep}. 

\noindent{\em Proof of Theorem \ref{pdetrep}.} Let us denote $\begin{bmatrix}
    w_1 \cr w_2
\end{bmatrix} =\Phi ( \begin{bmatrix}
    z_1 \cr z_2
\end{bmatrix})$.
Put
$$ \tilde{p} (z_1, z_2) = p(\frac{i(1-z_1^2-z_2^2)}{(1-z_1)^2+z_2^2}, \frac{2z_2}{(1-z_1)^2+z_2^2})({(1-z_1)^2+z_2^2})^{n_1+n_2} . 
$$
Then, for $(z_1,z_2) \in L_2$, we have
$$ |\tilde{p} (z_1,z_2)| = |p(\frac{i(1-z_1^2-z_2^2)}{(1-z_1)^2+z_2^2}, \frac{2z_2}{(1-z_1)^2+z_2^2})||({(1-z_1)^2+z_2^2})^{n_1+n_2}| =$$ $$ |p(w_1,w_2)| \left|\left(1-\frac{1+w_1^2-w_2^2}{(w_1+i)^2-w_2^2} \right)^2 + \left(\frac{-2w_2}{(w_1+i)^2-w_2^2} \right)^2 \right|^{n_1+n_2} \ge \epsilon,$$ 
due to the condition on $p$ as $(w_1, w_2)  \in T_{{\mathcal C}_2}$. Thus $\tilde{p}$ is strongly stable on $\overline{L_2}$. By applying Theorem \ref{2ball} there exists a polynomial $\tilde{q}$, a nonnegative integer $k$, and a strictly contractive matrix $K$ of size $ 2k \times 2k$ so that 
$$\tilde{p}\tilde{q} (z_1,z_2)  =  \det(I-K (P(z_1,z_2)\otimes I_k)).$$
Let us denote $W=\begin{bmatrix}
    w_1 & -iw_2\cr iw_2 & w_1
\end{bmatrix}$.
Then, using Lemma \ref{philemma} we have for $(w_1,w_2) \in T_{{\mathcal C}_2}$, 
$$ \tilde{p}\tilde{q} (\frac{1+w_1^2-w_2^2}{(w_1+i)^2-w_2^2}, \frac{-2w_2}{(w_1+i)^2-w_2^2} )((w_1+i)^2-w_2^2)^{k} = $$ 
\begin{equation}\label{ff} \det\left(I-K(\phi^{-1}(W) \otimes I_k) \right) \det(W+iI_2)^k= \end{equation}
$$ \det \left( (W+iI_2)\otimes I_k + K ((W-iI_2)\otimes I_k) \right)=  $$
$$ \det \left( i (I_{2k}-K) + (I_{2k}+K) (W \otimes I_k) \right) = $$
$$ \det\left( i(I_{2k}-K)(I_{2k}+K)^{-1} + W\otimes I_k \right) \det(I_{2k}+K). $$
Letting $A_0= i(I_{2k}-K)(I_{2k}+K)^{-1}$ and dividing both sides by $\det(I_{2k}+K)$ we obtain the desired determinantal representation.
\hfill $\Box$

\section{Strict Stability with respect to the $n$-variable Lorentz cone}\label{sec6}

In this section we present a determinantal representation theorem for the tube domain $T_{{\mathcal C}_n}$ that correspond to the $n$-variable Lorentz cone. Let us first observe that $w\in T_{{\mathcal C}_n}$ if and only if 
$$ W(w):= \begin{bmatrix} w_1 & -iw_2 & -iw_3 & \cdots & -iw_n \cr iw_2 & w_1 & 0 & \cdots & 0 \cr iw_3 & 0 & w_1 & \cdots & 0 \cr \vdots & \vdots & \vdots & \ddots & \vdots \cr iw_n & 0 & 0 & \cdots & w_1 \end{bmatrix} $$
has positive definite imaginary part. Indeed, one computes ${\rm Im} \ W(w)$ as in \eqref{imW}, and subsequently take a Schur complement. The main objective in this section is to establish the following result.

\begin{theorem}\label{Lorentz} 
Let $p(w_1,\ldots , w_n)$ be a polynomial of multidegree $(n_1, \ldots, n_n)$ so that for some $\epsilon >0$ we have
that $(w_1, \ldots , w_n)  \in T_{{\mathcal C}_n}$ implies \begin{equation}\label{pbbdbelow3} |p(w_1, \ldots , w_n) | \ge \end{equation} $$\epsilon  \left|\left(1-\frac{1+w_1^2-\sum_{j=2}^n w_j^2}{(w_1+i)^2-\sum_{j=2}^n w_j^2} \right)^2 + \sum_{j=2}^n \left(\frac{-2w_j}{(w_1+i)^2-\sum_{j=2}^n w_j^2} \right)^2 \right|^{-\sum_{j=1}^n n_j}. $$ Then there exists a polynomial $q(w_1, \ldots , w_n)$ so that the product $pq$ has a certifying determinantal representation; that is, there exists $k\in{\mathbb N}$ and a $m\times m$ matrix $A_0$ with $m\le nk$ and ${\rm Im}\ A_0 \ge 0$ so that
$$ p(w)q(w) = (w_1+i)^k\det \left( A_0 + V^* (W(w) \otimes I_k) V \right), $$ where $W(w)$ is as above and $V$ is $nk \times m$ so that $V^*V=I_m$.
\end{theorem}

As in Section \ref{sec4} we will also first establish a determinantal representation with respect to the Lie ball $L_n$, but the description of $L_n$ will now require a rational matrix function description as opposed to the simple linear $P(z_1,z_2)$ description used in Section \ref{sec4}. We start this section with developing this rational matrix function description of $L_n$.

Define
$$ \Phi_n ((z_j)_{j=1}^n) := \frac{1}{(1-z_1)^2+z_2^2+\cdots +z_n^2} \begin{bmatrix} i(1-z_1^2-\cdots - z_n^2) \cr 2z_2 \cr \vdots \cr 2z_n \end{bmatrix}, $$
where $z_1,\ldots, z_n\in\CC$ are so that $(1-z_1)^2+z_2^2+\cdots +z_n^2\neq 0$.
\begin{prop}\label{Phinprop}
The map $\Phi_n: L_n \to T_{{\mathcal C}_n}$ is invertible, and its inverse is given by
$$ \Phi_n^{-1} ((w_j)_{j=1}^n )= \frac{1}{(w_1+i)^2-w_2^2-\cdots -w_n^2} \begin{bmatrix} 1+ w_1^2-w_2^2-\cdots - w_n^2 \cr -2w_2 \cr \vdots \cr -2w_n \end{bmatrix}. $$
\end{prop}

\begin{proof}
On ${\mathbb C}^n$ we define the multiplication via $z=uv$ where
$$ z_1=\sum_{j=1}^n u_jv_j ,\ \  z_j=u_1v_j+u_jv_1, j=2,\ldots , n.$$ The neutral element is $e=(1,0, \ldots, 0)$, and $u$ is invertible when $u_1^2-u_2^2-\cdots - u_n^2 \neq 0$, and in that case
$$ u^{-1}= \frac{1}{u_1^2-u_2^2-\cdots - u_n^2} (u_1,-u_2,\ldots , -u_n).$$
As an aside, we  mention that the above multiplication is not associative. E.g, for $f_2=(0,1,0,\ldots,0)$ and $f_3=(0,0,1,0,\ldots,0)$, then $(f_3 f_2)f_2=0\neq f_3 = f_3(f_2 f_2)$.
Using the above notation, we have that 
\begin{equation}\label{Cayley}
\Phi_n(z)=-(ie+iz)(iz-ie)^{-1}. \end{equation} 
The result now follows similarly as in the proof of \cite[Theorem X.4.3]{Faraut}.
\end{proof}

Equation \eqref{Cayley} shows that $\Phi_n$ is a Cayley type transform.
Using the above we arrive at the following description of the Lie ball.

\begin{prop}\label{Q} We have $z\in L_n$ if and only if 
$$ Q(z):=\begin{bmatrix} z_1 & -z_2 & -z_3 & \cdots & -z_n \cr z_2 & z_1-\frac{\sum_{j\neq 2,j=2}^n z_j^2}{1-z_1} & \frac{z_2z_3}{1-z_1} & \cdots & \frac{z_2z_n}{1-z_1} \cr z_3 & \frac{z_2z_3}{1-z_1} & z_1-\frac{\sum_{j\neq 3,j=2}^n z_j^2}{1-z_1} & \cdots & \frac{z_3z_n}{1-z_1}\cr \vdots & \vdots & \vdots & \ddots & \vdots \cr z_n &
\frac{z_2z_n}{1-z_1} & \frac{z_3z_n}{1-z_1}& \cdots & z_1-\frac{\sum_{j\neq n,j=2}^n z_j^2}{1-z_1}\end{bmatrix} $$ is a strict contraction.   
\end{prop}

\begin{proof}
Recall that $w=(w_1,\ldots , w_n) \in T_{{\mathcal C}_n}$ if and only if
$$ W(w):= \begin{bmatrix} w_1 & -iw_2 & -iw_3 & \cdots & -iw_n \cr iw_2 & w_1 & 0 & \cdots & 0 \cr iw_3 & 0 & w_1 & \cdots & 0 \cr \vdots & \vdots & \vdots & \ddots & \vdots \cr iw_n & 0 & 0 & \cdots & w_1 \end{bmatrix} $$ has positive definite imaginary part. Thus $w\in T_{{\mathcal C}_n}$ if and only if $ \phi^{-1}(W(w)) $ is a strict contraction.
Next, we have that $\phi^{-1}(W(w))=Y=(y_{kl})_{k.l=1}^n$, where
$$ y_{11}= \frac{1+w_1^2-w_2^2-\cdots -w_n^2}{(w_1+i)^2-w_2^2-\cdots -w_n^2}, $$ $$ y_{1j}=\frac{2w_j}{(w_1+i)^2-w_2^2-\cdots -w_n^2}, j=2,\ldots , n, $$
$$y_{j1}=\frac{-2w_j}{(w_1+i)^2-w_2^2-\cdots -w_n^2}, j=2,\ldots , n,$$ 
$$ y_{kl}= \frac{-2iw_kw_l}{(w_1+i)((w_1+i)^2-w_2^2-\cdots -w_n^2)}, k\neq l, k,l=2,\ldots, n,$$
and for $k=2,\ldots,n$,
$$ y_{kk} = \frac{w_1^3+iw_1^2+w_1+i-w_1(w_2^2+\ldots +w_n^2) +i \sum_{j\neq k,j=2}^n w_j^2 -i w_k^2}{(w_1+i)((w_1+i)^2-w_2^2-\cdots -w_n^2)}. $$
When we substitute $w=\Phi_n(z)$ in $\phi^{-1}(W(w))$ we find
$$ \phi^{-1}(W(\Phi_n(z))) = \ \ \ \ \ \ \ \ \ \ \ \ \ \ \ \ \ $$ $$\begin{bmatrix} z_1 & -z_2 & -z_3 & \cdots & -z_n \cr z_2 & z_1-\frac{\sum_{j\neq 2,j=2}^n z_j^2}{1-z_1} & \frac{z_2z_3}{1-z_1} & \cdots & \frac{z_2z_n}{1-z_1} \cr z_3 & \frac{z_2z_3}{1-z_1} & z_1-\frac{\sum_{j\neq 3,j=2}^n z_j^2}{1-z_1} & \cdots & \frac{z_3z_n}{1-z_1}\cr \vdots & \vdots & \vdots & \ddots & \vdots \cr z_n &
\frac{z_2z_n}{1-z_1} & \frac{z_3z_n}{1-z_1}& \cdots & z_1-\frac{\sum_{j\neq n,j=2}^n z_j^2}{1-z_1}\end{bmatrix}.$$
Combining these observations with Proposition \ref{Phinprop} yields the result.
\end{proof}

Let 
$$ S_r(z) = 1 \oplus (r-z_1)I_{n-1} $$
and
$$ T_r(z)= \begin{bmatrix} z_1 & -z_2(r-z_1) &  \cdots & -z_n(r-z_1) \cr z_2 & z_1(r-z_1)-{\sum_{j\neq 2,j=2}^n z_j^2}  & \cdots & {z_2z_n} \cr  \vdots & \vdots & \ddots & \vdots \cr z_n &
{z_2z_n} &  \cdots & z_1(r-z_1)-{\sum_{j\neq n,j=2}^n z_j^2}\end{bmatrix}.$$

\begin{cor}\label{Ppm} Let the $n\times n$  matrix polynomials $P_\pm(z)$ be given by 
$$ P_+(z) = 1\oplus (1-z_1)I_{n-1}, \ P_-(z)=Q(z) P_+(z).$$
Then $z\in L_n$ if and only if 
\begin{equation*}
 P_+(z)^*P_+(z)-P_-(z)^*P_-(z) >0.\end{equation*} 
\end{cor}

\begin{proof}
The statement is a direct consequence of Proposition \ref{Q}. 
\end{proof}

\begin{lemma}\label{Taylor}
    Let $T=(T_1,\ldots, T_n) $ be a commuting tuple of Hilbert space operators so that \begin{equation}\label{P(T)}
 P_+(T)^*P_+(T)-P_-(T)^*P_-(T) \ge 0.\end{equation} 
 Then $\spec(T) \subseteq \overline{L_n}$.
\end{lemma}

\begin{proof}
The $(1,1)$ entry of \eqref{P(T)} equals
$I-\sum_{j=1}^n T_j^*T_j \ge 0$, so in particular each $T_j$ is a contraction. For $r>1$, let $f_r(z)=\frac{1-z}{r-z}$. By von Neumann's inequality, we have that 
$$\| f_r(T_1)\|= \| (I-T_1)(rI-T_1)^{-1} \| \le \sup_{|z|\le 1} |f_r(z)| = \frac{2}{r+1}.$$ Now, letting $P_{+,r}(z) = 1\oplus (r-z_1)I_{n-1}$, we have that 
\begin{equation*}
 P_{+,r}(T)^*P_{+,r}(T)-P_+(T)^*P_+(T) \ge
 P_{+,r}(T)^*(0 \oplus (1-\frac{2^2}{(r+1)^2})I)P_{+,r}(T) \ge 0.\end{equation*} Thus
 \begin{equation*}
 P_{+,r}(T)^*P_{+,r}(T)-P_-(T)^*P_-(T)  \ge 0.\end{equation*}
 Since each $T_j$ is a contraction, $P_{+,r}(T)$ is invertible and we can rewrite this as
  \begin{equation*}
\|P_{+,r}(T)^{-1} P_-(T)\| \leq 1,\end{equation*}
and each entry of the rational matrix function $P_{+,r}(z)^{-1} P_-(z)$ is holomorphic on a neighbourhood of $\overline{\DD}^n \supseteq \spec(T)$. Taking a unit vector $v \in \CC^n$, we can now apply the spectral mapping theorem to the rational function 
$v^* P_{+,r}(z)^{-1} P_-(z) v$ to deduce that
\begin{equation*}
|v^* P_{+,r}(\lambda)^{-1} P_-(\lambda) v |\leq 1
\end{equation*}
for all $\lambda\in\spec(T)$,
since $v^* P_{+,r}(T)^{-1} P_-(T) v$ is a contraction. It follows that
\begin{equation*}
\spec(T) \subseteq \bigcap_{r>1} \left\{z \in \overline{\DD}^n \colon \|P_{+,r}(z)^{-1} P_-(z)\| \leq 1\right\} = \overline{L_n}.
\end{equation*}
\end{proof}

We prove a version of \cite[Lemma 3.2]{GKVW16a} in the current setting.
\begin{lemma}\label{Lm32}
    There exists a $n$-tuple $T_{\max}$ of commuting bounded linear operators on a Hilbert space satisfying 
    \begin{equation}\label{P+-} P_+(T_{\max})^*P_+(T_{\max})-P_-(T_{\max})^*P_-(T_{\max}) \geq 0 \end{equation} and
    $$ \| q (T_{\max} ) \| = \sup_{T \colon P_+(T)^*P_+(T)-P_-(T)^*P_-(T) \geq 0} \|q(T) \| =: \| q \|_{P_\pm}$$
    for every polynomial $q$.
\end{lemma}

\begin{proof} We follow the arguments in \cite[Page 65]{Paulsen}.
First note that for a ﬁxed polynomial $p$, there is always an $n$-tuple of commuting operators satisfying \eqref{P+-} where
this supremum is achieved. To see this, ﬁrst choose a sequence of $n$-tuples
$(T_1^{(k)}, \ldots , T_n^{(k)} )$ such that $\| p \|_{P_\pm} = \sup_k \| p(T_1^{(k)}, \ldots, T_n^{(k)} )\|$. If we then let $T_i =
\oplus_k
T_i^{(k)}$, $i = 1, \ldots , n$, be the $n$-tuple  on the direct
sum of the corresponding Hilbert spaces, then $\| p(T_1, \ldots, T_n )\| = \| p \|_{P_\pm}$.

Now, if for each polynomial $p$ we choose a commuting $n$-tuple where $\| p \|_{P_\pm}$ is attained and form the direct sum of all such $n$-tuples, then we
have a single $n$-tuple $(T_1, \ldots, T_n )$ on a Hilbert space
${\mathcal H}$ such that $\| p \|_{P_\pm} = \| p(T_1, \ldots , T_n ) \|$ for every polynomial.
\end{proof}
\noindent

\begin{theorem}\label{nball} 
Let $p(z)$ have no roots on the closed $n$-variable Lie ball $\overline{L_n}$. Then there exists a polynomial $q(z)$, a positive integer $k$ and a $nk \times nk$ contractive matrix $K$ so that 
$$ p(z) q(z) = \det((P_+(z)\otimes I_k) - K (P_-(z) \otimes I_{k})), $$
where $P_\pm(z)$ are defined in Corollary \ref{Ppm}.
\end{theorem}

\noindent 
\begin{proof} We will apply a variation of \cite[Theorem 4.1]{GKVW16a} where the inequality $I-P(z)^*P(z)>0 $ is replaced by $P_+(z)^*P_+(z)-P_-(z)^*P_-(z) >0$.

We first need to observe that for any analytic matrix valued function $F$ on a neighbourhood of $\overline{L_n}$, we have that 
\begin{equation}\label{finiteAgler} \sup_{T \colon P_+(T)^*P_+(T)-P_-(T)^*P_-(T) \geq 0} \|F(T) \| < \infty. \end{equation} 
The argument is the same as in the proof of \cite[Lemma 3.3]{GKVW16a}. For this we need that $\overline{L_n}$ is polynomially convex, which follows from the convexity of $\overline{L_n}$ (see, e.g., \cite[Page ix]{Loos}). The other parts in the argument are covered by Lemmas \ref{Taylor} and \ref{Lm32}.

Now we can proceed as in the proof of \cite[Theorem 3.4]{GKVW16a}, where \cite[Theorem 2.3]{GKVW16a} is applied with $P_j=P_+^*P_+-P_-^*P_- $. At the end of the proof of \cite[Theorem 3.4]{GKVW16a} one now finds the equations
$$ A(P_-(z)\otimes I_k) H(z) + B(z) R(z) = (P_+(z)\otimes I_k)H(z), $$ $$C(P_-(z)\otimes I_k) H(z) + D(z) R(z) = Q(z).$$ Eliminating $H(z)$, then gives $$ Q(z)R(z)^{-1} = D + C(P_-(z)\otimes I_k) ((P_+(z)\otimes I_k)-A(P_-(z)\otimes I_k))^{-1}B.$$ Adjusting now the proof of \cite[Theorem 4.1]{GKVW16a} yields the desired determinantal representation.
\end{proof}

\noindent {\em Proof of Theorem \ref{Lorentz}} We proceed as in the proof of Theorem \ref{pdetrep}. Let us denote \begin{equation}\label{PHI} \begin{bmatrix}
    w_1 \cr \vdots \cr w_n
\end{bmatrix} =\Phi_n ( \begin{bmatrix}
    z_1 \cr \vdots \cr z_n
\end{bmatrix}).\end{equation}
Put
$$ \tilde{p} (z_1, \ldots , z_n) =$$ $$ p(\frac{i(1-z_1^2-\sum_{j=2}^n z_j^2)}{(1-z_1)^2+\sum_{j=2}^n z_j^2}, \frac{2z_2}{(1-z_1)^2+\sum_{j=2}^n z_j^2}, \cdots ,\frac{2z_n}{(1-z_1)^2+\sum_{j=2}^n z_j^2} ) \times $$ $$({(1-z_1)^2+\sum_{j=2}^n z_j^2})^{\sum_{j=1}^n n_j} . 
$$
Then, similar as in the proof of Theorem \ref{pdetrep} one deduces from the condition \eqref{pbbdbelow3} that for $(z_1,\ldots , z_n) \in L_n$, we have
$$ |\tilde{p} (z_1,\ldots , z_n)|  \ge \epsilon.$$ 
Thus $\tilde{p}$ is strongly stable on $\overline{L_n}$. By the assumption that Conjecture \ref{nball} holds there exists a polynomial $\tilde{q}(z)$, a positive integer $k$ and a $nk \times nk$ contractive matrix $K$ so that 
$$ \tilde{p}(z) \tilde{q}(z) = \det((P_+(z)\otimes I_k) - K (P_-(z) \otimes I_{k})). $$
Then 
$$ \tilde{p}(z) \tilde{q}(z) = \det(I - K (P_-(z)P_+(z)^{-1} \otimes I_{k}))(\det P_+(z))^k. $$
Since $K$ is a contraction, if $K$ has an eigenspace at eigenvalue 1 it is a reducing subspace. Thus there exists a unitary $U$ so that
\begin{equation}\label{uu} K = U \begin{bmatrix} I & 0 \cr 0 & \tilde{K} \end{bmatrix} U^*,\end{equation}
where $\tilde{K}$ is an $m\times m$ ($m\le nk$) contraction so that $I-\tilde{K}$ is invertible.

For $w \in T_{{\mathcal C}_n}$, 
$$ \tilde{p}\tilde{q} (\frac{1+w_1^2-\sum_{j=2}^n w_j^2}{(w_1+i)^2-\sum_{j=2}^n w_j^2}, \frac{-2w_2}{(w_1+i)^2-\sum_{j=2}^n w_j^2}, \ldots, \frac{-2w_n}{(w_1+i)^2-\sum_{j=2}^n w_j^2} ) $$ $$ \times ((w_1+i)^2-\sum_{j=2}^n w_j^2)^{nk} = $$ 
\begin{equation*}\label{ff} \det((P_+(\Phi_n^{-1}(w))\otimes I_k) - K (P_-(\Phi_n^{-1}(w)) \otimes I_{k}))((w_1+i)^2-\sum_{j=2}^n w_j^2)^{nk}= \end{equation*}
$$ \det (I-K (\phi^{-1}(W(w))\otimes I_k)) \det(P_+(\Phi_n^{-1}(w)))^k((w_1+i)^2-\sum_{j=2}^n w_j^2)^{nk}= $$
$$ \det (((W(w)+iI)\otimes I_k)-K ((W(w)-iI)\otimes I_k)) \det(W(w)+iI)^{-k} $$ $$
\times \left( \frac{2i(w_1+i)}{(w_1+i)^2-\sum_{j=2}^n w_j^2} \right)^{(n-1)k} ((w_1+i)^2-\sum_{j=2}^n w_j^2)^{nk}= $$
$$ \det \left( (W(w)+iI)\otimes I_k - K ((W(w)-iI)\otimes I_k) \right)(2i)^{(n-1)k} (w_1+i)^k=  $$
\begin{equation}\label{ff2} \det \left( i(I+K) + (I-K)(W(w) \otimes I_k)\right) (2i)^{(n-1)k} (w_1+i)^k. \end{equation}
Let us write $U$ in \eqref{uu} as
$$ U = \begin{bmatrix}
    \tilde{U} & V 
\end{bmatrix},$$ where $V$ is an isometry (as $U$ is unitary) of size $nk\times m$. Substituting \eqref{uu} into \eqref{ff2} we obtain
$$ \det \left( i \begin{bmatrix}
    2I & 0 \cr 0 & I+\tilde{K}
\end{bmatrix} + \begin{bmatrix}
    0 & 0 \cr 0 & I-\tilde{K}
\end{bmatrix} U^* (W(w) \otimes I_k) U\right) (2i)^{(n-1)k} (w_1+i)^k=$$
$$ (2i)^{nk-m} \det \left( i (I-\tilde{K})^{-1}(I+\tilde{K}) + V^* (W(w) \otimes I_k) V\right) \times $$ $$ \det(I-\tilde{K})(2i)^{(n-1)k} (w_1+i)^k.$$ Letting $A_0 = i (I-\tilde{K})^{-1}(I+\tilde{K})$ and rearranging the constants we derive the desired result.
\hfill $\Box$

\eject
\section{Exceptional tube domain}\label{sec7} According to \cite{Yin} the exceptional tube domain ${\mathcal{T}}_{27}$ can be described as follows.
Let $Y(\omega)$, $\omega =(\omega_j)_{j=1}^8 \in {\mathbb C}^8$, be defined by
$$ Y(\omega) = \begin{bmatrix}
    \omega_1 & \omega_2 & \omega_3 & \omega_4 & \omega_5 & \omega_6 & \omega_7 & \omega_8 \cr 
    \omega_2 & -\omega_1 & -\omega_4 & \omega_3 & -\omega_6 & \omega_5 & \omega_8 & -\omega_7 \cr 
    \omega_3 & \omega_4 & -\omega_1 & -\omega_2 & -\omega_7 & -\omega_8 & \omega_5 & \omega_6 \cr 
    \omega_4 & -\omega_3 & \omega_2 & -\omega_1 & -\omega_8 & \omega_7 & -\omega_6 & \omega_5 \cr 
    \omega_5 & \omega_6 & \omega_7 & \omega_8 & -\omega_1 & -\omega_2 & -\omega_3 & -\omega_4 \cr 
    \omega_6 & -\omega_5 & \omega_8 & -\omega_7 & \omega_2 & -\omega_1 & \omega_4 & -\omega_3 \cr 
    \omega_7 & -\omega_8 & -\omega_5 & \omega_6 & \omega_3 & -\omega_4 & -\omega_1 & \omega_2 \cr 
    \omega_8 & \omega_7 & -\omega_6 & -\omega_5 & \omega_4 & \omega_3 & -\omega_2 & -\omega_1 \cr 
\end{bmatrix}.$$ When we write $Y(\omega) =\sum_{j=1}^8 \omega_j T_j$ we have that 
\begin{equation}\label{Ts}
 T_k^TT_j + T_j^TT_k=2\delta_{jk}I_8, \end{equation}
where $\delta_{jk}$ is the Kronecker delta. The tube domain ${\mathcal{T}}_{27}$ is now given by
$$ {\mathcal{T}}_{27}= \left\{ W=\begin{bmatrix} w_{11} & w_{12}^T & w_{13}^T \cr w_{12} & w_{22}I_8 & Y(w_{23}) \cr w_{13} & Y(w_{23})^T & w_{33}I_8 \end{bmatrix} : {\rm Im}\ W >0 \right\},
$$ where $w_{jj} \in {\mathbb C}$, $j=1,2,3$, and $w_{12},w_{13}, w_{23} \in {\mathbb C}^8$. 
\begin{lemma}\label{switch} For $y,z\in {\mathbb C}^8$ we have $Y(z)y=T_1Y(y)z$. Also, for $x\in{\mathbb R}^n$, we have
$$ \begin{bmatrix}
    \alpha  & x^T \cr
x & \beta I_n    \end{bmatrix} >0 $$ if and only if 
$$ \begin{bmatrix}
    \beta  & x^T \cr
x & \alpha I_n   \end{bmatrix} >0 $$
\end{lemma}


\begin{lemma}\label{Yw} Let $Y(\omega)$ be as above. Then $$Y(\omega)^TY(\omega) = (\sum_{j=1}^8 \omega_j^2)I_8 = Y(\omega)Y(\omega)^T. $$
\end{lemma}

\begin{proof}
Write $Y(\omega)=\sum_{j=1}^8 \omega_j T_j$. Using now \eqref{Ts}, the lemma follows easily. 
\end{proof}

We now have three equivalent ways to describe ${\mathcal{T}}_{27}$, as follows.
\begin{prop}\label{T27} For $w_{jj} \in {\mathbb C}$, $j=1,2,3$, and $w_{12},w_{13}, w_{23} \in {\mathbb C}^8$, we have
$${\rm Im} \begin{bmatrix} w_{11} & w_{12}^T & w_{13}^T \cr w_{12} & w_{22}I_8 & Y(w_{23}) \cr w_{13} & Y(w_{23})^T & w_{33}I_8 \end{bmatrix} >0$$ if and only if
$${\rm Im} \begin{bmatrix} w_{22} & w_{12}^T & w_{23}^T \cr w_{12} & w_{11}I_8 & T_1Y(w_{13}) \cr w_{23} & Y(w_{13})^TT_1 & w_{33}I_8 \end{bmatrix} >0$$ if and only if
$${\rm Im} \begin{bmatrix} w_{33} & w_{23}^T & w_{13}^T \cr w_{23} & w_{22}I_8 & Y(T_1w_{12}) \cr w_{13} & Y(T_1w_{12})^T & w_{11}I_8 \end{bmatrix} >0.$$
\end{prop}

\begin{proof}
    First we observe that  
    $${\rm Im} \begin{bmatrix} w_{11} & w_{12}^T & w_{13}^T \cr w_{12} & w_{22}I_8 & Y(w_{23}) \cr w_{13} & Y(w_{23})^T & w_{33}I_8 \end{bmatrix} =  \begin{bmatrix} {\rm Im}\ w_{11} & {\rm Im}\ w_{12}^T & {\rm Im}\ w_{13}^T \cr {\rm Im}\ w_{12} & {\rm Im}\ w_{22}I_8 & Y({\rm Im}\ w_{23}) \cr {\rm Im}\ w_{13} & Y({\rm Im}\ w_{23})^T & {\rm Im}\ w_{33}I_8 \end{bmatrix}.$$
    Let now $y_{ij}={\rm Im}\ z_{ij}$, $i,j=1,2,3$. The above equivalences, now follow from the following type of Schur complement calculations:
    $$\begin{bmatrix} y_{11} & y_{12}^T & y_{13}^T \cr y_{12} & y_{22}I_8 & Y(y_{23}) \cr y_{13} & Y(y_{23})^T & y_{33}I_8 \end{bmatrix} >0$$ if and only if $y_{33}>0$ and
    $$\begin{bmatrix} y_{11} & y_{12}^T \cr y_{12} & y_{22}I_8  \end{bmatrix} - \frac{1}{y_{33}} \begin{bmatrix}  y_{13}^T \cr  Y(y_{23})  \end{bmatrix} \begin{bmatrix}  y_{13} & Y(y_{23})^T  \end{bmatrix} >0.$$ Using now that $Y(y_{23})y_{13}=T_1Y(y_{13})y_{23}$, $Y(y_{23})Y(y_{23})^T =y_{23}^Ty_{23} I_8$, and the switch in diagonal entries as in Lemma \ref{switch} gives
    $$\begin{bmatrix} y_{22}-\frac{1}{y_{33}}y_{23}^Ty_{23} & y_{12}^T-\frac{1}{y_{33}} y_{23}^TY(y_{13})^TT_1 \cr y_{12}-\frac{1}{y_{33}} T_1Y(y_{13})y_{23}& (y_{11}-\frac{1}{y_{33}}y_{13}^Ty_{13})I_8  \end{bmatrix} >0.$$ The latter is equivalent to 
    $$ \begin{bmatrix} y_{22} & y_{12}^T & y_{23}^T \cr y_{12} & y_{11}I_8 & T_1Y(y_{13}) \cr y_{23} & Y(y_{13})^TT_1 & y_{33}I_8 \end{bmatrix} >0,$$
    where we used that $T_1^2=I$. This proves the equivalence of the first two inequalities. The proof of the equivalence with the third inequality is similar.
\end{proof}

For our theory to work we need to identify a bounded domain whose Cayley transform yields ${\mathcal{T}}_{27}$. This leads to the following definition. For $\zeta=(w_1,x,y,w_2, z,w_3)\in
({\mathbb C}\setminus \{ 1 \} ) \times {\mathbb C}^8\times {\mathbb C}^8\times {\mathbb C}\times {\mathbb C}^8\times {\mathbb C} $, let 
\begin{equation}\label{Xzeta}
X (\zeta) = \begin{bmatrix}
    w_1 & x^T & y^T \cr x & w_2I_8 & Y(z) \cr y & Y(z)^T & w_3 I_8
\end{bmatrix} - \frac{1}{1-w_1} \begin{bmatrix}
    0 \cr x \cr y  
\end{bmatrix}\begin{bmatrix}
    0 & x^T & y^T  
\end{bmatrix}. \end{equation} Define now the domain
$$ {\mathcal C} = \{ \zeta : \| X(\zeta ) \| <1  \}.$$

\begin{prop}\label{71}
    The Cayley transform $\phi(X)=i(I+X)(I-X)^{-1}$ maps 
${\mathcal C}$ bijectively onto ${\mathcal{T}}_{27}$.
\end{prop}

\begin{lemma}\label{phi2x2}
If $X=\begin{bmatrix} w & \psi^T \cr \psi & Z-\frac{1}{1-w}\psi \psi^T \end{bmatrix}$, then 
\begin{equation}\label{phiX}
    \phi(X)= \begin{bmatrix}
    \phi(w)+\frac{2i}{(1-w)^2} \psi^T (I-Z)^{-1} \psi  & \frac{2i}{1-w} \psi^T (I-Z)^{-1}  \\
    \frac{2i}{1-w} (I-Z)^{-1} \psi & \phi(Z)
  \end{bmatrix}.
\end{equation}

\end{lemma}
\begin{proof}
 We use the well known formula for inverses of block matrices:   
$$\begin{bmatrix}
    {A} & {B} \\
    {C} & {D}
  \end{bmatrix}^{-1} = \begin{bmatrix}
     {A}^{-1} + {A}^{-1}{B}\left({D} - {CA}^{-1}{B}\right)^{-1}{CA}^{-1} &
      -{A}^{-1}{B}\left({D} - {CA}^{-1}{B}\right)^{-1} \\
    -\left({D}-{CA}^{-1}{B}\right)^{-1}{CA}^{-1} &
       \left({D} - {CA}^{-1}{B}\right)^{-1}
  \end{bmatrix}.$$
  Computing $(I-X)^{-1}$ we find
  $$(I-X)^{-1}= \begin{bmatrix}
    \frac{1}{1-w} + \frac{1}{(1-w)^2}\psi^T(I-Z)^{-1}\psi & \frac{1}{1-w}\psi^T(I-Z)^{-1} \\
    \frac{1}{1-w}(I-Z)^{-1}\psi & (I-Z)^{-1}
  \end{bmatrix}.$$ Multiplying on the left with $i(I+X)$ it is now easy to obtain \eqref{phiX}.
\end{proof}

\noindent {\em Proof of Proposition \ref{71}.} Let $X(\zeta) \in {\mathcal C}$.
It follows from Lemma \ref{phi2x2} that 
\begin{equation}\label{phiX2}  
\phi(X(\zeta)) = \begin{bmatrix} * & * \cr * & \phi \left( \begin{bmatrix}
     w_2I_8 & Y(z) \cr  Y(z)^T & w_3 I_8
\end{bmatrix} \right) \end{bmatrix} . \end{equation} 
Using Lemma \ref{Yw} and a $2\times 2$ block matrix inverse formula using Schur complements, one sees that 
$$ \begin{bmatrix}
    (1- w_2)I_8 & -Y(z) \cr  -Y(z)^T & (1-w_3) I_8
\end{bmatrix}^{-1} = \begin{bmatrix}
     \alpha I_8 & \gamma Y(z) \cr  \gamma Y(z)^T & \beta I_8
\end{bmatrix},$$
where $$ \alpha = \frac{(1-w_3)}{(1-w_2)(1-w_3) -\|z\|^2}, \beta= \frac{(1-w_2)}{(1-w_2)(1-w_3) -\|z\|^2}, $$ $$\gamma =\frac{1}{(1-w_2)(1-w_3) -\|z\|^2}.$$
But, then 
$$ \phi \left( \begin{bmatrix}
     w_2I_8 & Y(z) \cr  Y(z)^T & w_3 I_8
\end{bmatrix} \right) = \begin{bmatrix}
     \alpha' I_8 & \gamma' Y(z) \cr  \gamma' Y(z)^T & \beta' I_8
\end{bmatrix},$$
where $\begin{bmatrix}
     \alpha' & \gamma' \cr  \gamma' & \beta'
\end{bmatrix} = \phi\left( \begin{bmatrix}
     w_2 & \|z\| \cr  \|z\| & w_3
\end{bmatrix}\right).$
This gives that $\phi(X(\zeta))$ has the required format to be an element of ${\mathcal T}_{27}$. As $\| X(\zeta) \| <1$, we of course also get that ${\rm Im} \ \phi(X(\zeta)) >0$.
\hfill $\Box$

As an aside, we mention  
$$\det( \begin{bmatrix}
     (1-w_2)I_8 & -Y(z) \cr  -Y(z)^T & (1-w_3) I_8
\end{bmatrix}) = ((1-w_2)(1-w_3)-z^Tz)^8$$
and 
$$\det(I-X(\zeta))=(1-w_1)((1-w_2)(1-w_3)-z^Tz)^8.$$

The domain ${\mathcal C}$ is bounded, as we will see next. In fact, we have the following stronger condition.

\begin{lemma} Let $T=(T_j)_{j=1}^{27}$ be commuting Hilbert space operators so that $\| X(T) \| <1, $ where $X(\zeta)$ is defined in \eqref{Xzeta}. Then each $T_j$, $j=1,\ldots, 27$, is a strict contraction.
\end{lemma}

\begin{proof} Suppose that $\| X(T) \| <1.$
Since the first block column of $X(T)$ is a strict contraction, it follows that  $T_j$, $j=1,\ldots, 17$, are strict contractions. In addition, we have that $B:= \phi(X(T))$ has positive definite imaginary part. But then the same is true when we remove the first block row and columns of $B$, yielding $\widetilde{B}$ with ${\rm Im} \ \widetilde{B} >0$. Thus $\| \phi^{-1} (\widetilde{B}) \| <1$. By \eqref{phiX2}, we have that 
$$ \phi^{-1} (\widetilde{B}) = \begin{bmatrix}
    I_8 \otimes T_{18} & Y((T_j)_{j=19}^{26}) \cr 
    Y((T_j^T)_{j=19}^{26})^T & I_8 \otimes T_{27}
\end{bmatrix}, $$
and thus we find that $T_j$, $j=18,\ldots, 27$, are also strict contractions.
\end{proof}

There is one result that we are missing to complete our argument. We state it as a conjecture.

\begin{conj}\label{polcon}
    The  domain ${\mathcal C}$ is polynomially convex.
\end{conj}
In fact, we believe that ${\mathcal C}$ is the 27-dimensional exceptional Cartan domain (and thus convex, which implies polynomially convexity). It seems that it would suffice to prove that ${\mathcal C}$ is rotationally invariant (i.e., $\zeta \in {\mathcal C}$ implies $e^{i\theta} \zeta \in {\mathcal C}$) and subsequently use the theory in \cite{Loos} to conclude that ${\mathcal C}$ is the 27-dimensional exceptional Cartan domain.

We define $\Omega: {\mathbb C}^{27} \to \oplus_{j=1}^3 {\mathbb C}^{17\times 17}$ by $$\Omega (w) = \Omega_1 (w)\oplus \Omega_2 (w) \oplus \Omega_3 (w), \ \ w=(w_j)_{j=1}^{27} \in {\mathbb C}^{27}, $$ where
$$\Omega_1 (w) = \begin{bmatrix} w_{1} & ((w_j)_{j=2}^9)^T & ((w_j)_{j=10}^{17})^T \cr (w_j)_{j=2}^9 & w_{18}I_8 & Y((w_j)_{j=19}^{26}) \cr (w_j)_{j=10}^{17} & Y((w_j)_{j=19}^{26})^T & w_{27}I_8 \end{bmatrix},$$
$$\Omega_2 (w) = \begin{bmatrix} w_{18} & ((w_j)_{j=2}^9)^T & ((w_j)_{j=19}^{26})^T \cr (w_j)_{j=2}^9 & w_{1}I_8 & T_1Y((w_j)_{j=10}^{17}) \cr (w_j)_{j=19}^{26} & Y((w_j)_{j=10}^{17})^TT_1 & w_{27}I_8 \end{bmatrix},$$
$$\Omega_3 (w) = \begin{bmatrix} w_{27} & ((w_j)_{j=19}^{26})^T & ((w_j)_{j=10}^{17})^T \cr (w_j)_{j=19}^{26} & w_{18}I_8 & Y(T_1(w_j)_{j=2}^{9}) \cr (w_j)_{j=10}^{17} & Y(T_1(w_j)_{j=2}^{9})^T & w_{1}I_8 \end{bmatrix}.
\footnote{While we can describe the tube domain ${\mathcal T}_{27}$ by using only one of $\Omega_j$, $j=1,2,3$, the inclusion of all three direct summands is done so that in applying the results in \cite{GKVW16a} one can easily check the  Archimedean peroperty of the associated matrix system of Hermitian quadratic modules.}$$
Next, let $\eta(w)\in {\mathbb C}^{27}$ be defined via
$$ X(\eta(w))=\phi^{-1}(\Omega(w)),$$
which is defined for $w\in{\mathbb C}^{27}$ with $\det (I-\Omega(w))\neq 0$.
Then $\eta$ and its inverse are  rational vector valued maps. Let $r(w)$ be a lowest total degree polynomial so that 
$r(w) \eta^{-1}(w)$ is a vector valued polynomial map. We have the following conjecture.

\begin{conj}\label{except} 
Let $p(w_1,\ldots , w_{27})$ be a polynomial of multidegree $(n_1, \ldots, n_n)$ so that for some $\epsilon >0$ we have
that $(w_1, \ldots , w_{27})  \in {\mathcal T}_{27}$ implies \begin{equation}\label{pbbdbelow4} |p(w_1, \ldots , w_{27}) | \ge \epsilon  \left| r(\eta^{-1}(w))\right|^{-\sum_{j=1}^{27} n_j}.\end{equation} Then there exists a polynomial $q(w_1, \ldots , w_{27})$ so that the product $pq$ has a certifying determinantal representation; that is, there exists $k\in{\mathbb N}$ and a $m\times m$ matrix $A_0$ with $m\le nk$ and ${\rm Im}\ A_0 \ge 0$ so that
$$ p(w)q(w) = (w_1+i)^k\det \left( A_0 + V^* (\Omega(w) \otimes I_k) V \right), $$ where $\Omega(w)$ is as above and $V$ is $nk \times m$ so that $V^*V=I_m$.
\end{conj}

The proof for this conjecture is envisioned to go along the lines of the proof of Theorem \ref{Lorentz}. The main hurdle is the polynomially convexity of ${\mathcal C}$.


\bigskip

\subsection*{Acknowledgment} The authors wish to
thank the Banff International Research Station for the opportunity to work on this project
during the research in teams program 24rit023.

\end{document}